\documentclass[a4paper,oneside,12pt]{article}

\usepackage{amsmath,amsthm,amsfonts,amscd,amssymb}
\usepackage{longtable,geometry}
\usepackage[english]{babel}
\usepackage[utf8]{inputenc}
\usepackage[active]{srcltx}
\usepackage[T1]{fontenc}
\usepackage{graphicx}
\usepackage{pstricks}
\usepackage{bbm}
\usepackage{mathtools}
\usepackage{MnSymbol}
\usepackage{stmaryrd}
\usepackage{nicefrac}
\usepackage{calrsfs}
\usepackage{rotating}
\usepackage{xcolor}
\usepackage{framed}
\usepackage{hyperref}
\hypersetup{
    colorlinks=false,
    pdfborder={0 0 0},
}
\usepackage{enumitem}
\usepackage{footnote}

\colorlet{shadecolor}{blue!15}

\newtheorem{thm}{Theorem}[section]

\newtheorem{cor}[thm]{Corollary}
\newtheorem{lem}[thm]{Lemma}
\newtheorem{prop}[thm]{Proposition}

\newtheorem{question}[thm]{Question}

\newcommand{\be}[1]{\begin{equation}\label{#1}}
\newcommand{\ee}{\end{equation}}
\numberwithin{equation}{section}

\newcommand{\ba}[1]{\begin{align}\label{#1}}
\newcommand{\ea}{\end{align}}
\numberwithin{equation}{section}

\newcommand{\ben}{\begin{equation*}}
\newcommand{\een}{\end{equation*}}
\numberwithin{equation}{section}

\renewenvironment{proof}[1][\relax]
  {\paragraph{Proof\ifx#1\relax\else~of #1\fi}}%
  {~\hfill$\square$\par\bigskip}
  
\usepackage{psfrag} 
\def\mik{1}
\newcommand\cpsfrag[2]{\ifnum\mik=1\psfrag{#1}{#2}\fi}

\newcommand{\calC}{\mathcal{C}}
\newcommand{\calD}{\mathcal{D}}

\newcommand{\bbH}{\mathbb{H}}

\newcommand{\bbR}{\mathbb{R}}

\newcommand{\bbT}{\mathbb{T}}

\newcommand{\sfD}{{\sf D}}

\newcommand{\bfx}{\boldsymbol x}
\newcommand{\bfp}{\boldsymbol p}

\newcommand{\eps}{\varepsilon}

\newcommand{\rk}[1]{\bgroup\color{red}%
  \par\medskip\hrule\smallskip%
  \noindent\textbf{#1}%
  \par\smallskip\hrule\medskip\egroup}

\newcommand{\xlra}{\xleftrightarrow}

\newcommand{\ind}{\mathbf{1}}

\renewcommand{\int}{\mathrm{in}}

\usepackage{enumitem}
\setlist[itemize]{itemsep=1pt, topsep=4pt}
\setlist[enumerate]{itemsep=1pt, topsep=4pt}

\renewenvironment{thebibliography}[1]{
  \begin{oldthebibliography}{#1}
    \setlength{\itemsep}{0.25em}
}
{
  \end{oldthebibliography}
}

\title{Exponential decay in the loop~$O(n)$ model on the hexagonal lattice for~$n> 1$ and~$x<\tfrac{1}{\sqrt{3}}+\varepsilon(n)$}
\date{\today}
\author{Alexander Glazman\thanks{School of Mathematical Sciences,
Tel Aviv University,
Tel Aviv 69978,
Israel. \url{glazman@tauex.tau.ac.il}}
~and Ioan Manolescu\thanks{D\'epartement de Math\'ematiques, 
Universit\'e de Fribourg, 
23 Chemin du Mus\'ee, 
CH-1700 Fribourg, 
Switzerland.
\url{ioan.manolescu@unifr.ch}}}

\begin{document}

\maketitle

\begin{abstract}
We show that the loop~$O(n)$ model on the hexagonal lattice exhibits exponential decay of loop sizes 
whenever~$n> 1$ and~$x<\tfrac{1}{\sqrt{3}}+\varepsilon(n)$, for some suitable choice of~$\varepsilon(n)>0$. 

It is expected that, for~$n \leq 2$, the model exhibits a phase transition in terms of~$x$, that separates regimes of polynomial and exponential decay of loop sizes. In this paradigm, our result implies that the phase transition for~$n \in (1,2]$ occurs at some critical parameter~$x_c(n)$ strictly greater than that ~$x_c(1) = 1/\sqrt3$. 
The value of the latter is known since the loop~$O(1)$ model on the hexagonal lattice represents the contours of the spin-clusters of the Ising model on the triangular lattice. 

The proof is based on developing~$n$ as~$1+(n-1)$ and exploiting the fact that, when~$x<\tfrac{1}{\sqrt{3}}$, the Ising model exhibits exponential decay on any (possibly non simply-connected) domain. The latter follows from the positive association of the FK-Ising representation.
\end{abstract}

\section{Introduction}\label{sec:intro}

The loop~$O(n)$ model was introduced in~\cite{DomMukNie81} as a graphical model expected to be in the same universality class as the spin~$O(n)$ model. 
The latter is a generalisation of the seminal Ising model~\cite{Len20} that incorporates spins contained on the~$n$-dimensional sphere. See~\cite{PelSpi17} for a survey of both~$O(n)$ models. 
For integers~$n> 1$, the connection between the loop and the spin~$O(n)$ models remains purely heuristic. 
Nevertheless, the loop~$O(n)$ model became an object of study in its own right; it is predicted to have a rich phase diagram~\cite{BloNie89} in the two real parameters~$n,x>0$. For~$n=0,1,2$ the loop~$O(n)$ model is closely related to self-avoiding walk, the Ising model, and a certain random height model, respectively.

Let~$\bbH$ denote the hexagonal lattice. 
A \emph{domain} is a subgraph~$\calD = (V_\calD,E_\calD)$ of~$\bbH$ formed of the edges contained inside or along some simple cycle~$\partial\calD\subset E(\bbH)$ (hereafter called a \emph{loop}), 
and all endpoints of such edges. 
Write~$F_\calD$ for the set of faces of~$\bbH$ delimited by edges of~$\calD$ only.

Configurations $\omega \in \{0,1\}^{E_\calD}$ will be identified to the subset of edges $e \in E_\calD$ with $\omega(e) = 1$ 
(also called open edges) as well as to the spanning subgraph of $\calD$ containing these edges. 
A \emph{loop configuration} is any element~$\omega \in \{0,1\}^{E_\calD}$ that is even,
which is to say that the degree of any vertex is~$0$ or~$2$ when~$\omega$ is seen as a subgraph of~$\calD$.
As such~$\omega$ is the disjoint union of a set of loops of~$\calD$.
Loops are allowed to run along the boundary edges, but may not terminate at boundary points. 

For real parameters~$n,x >0$, let~${\sf Loop}_{\calD,n,x}$ be the measure on loop configurations given by 
\begin{align*}
	{\sf Loop}_{\calD,n,x}(\omega) = \frac1{Z_{\mathrm{loop}}(\calD,n,x)}\cdot x^{|\omega|}n^{\ell(\omega)},
\end{align*}
where~$|\omega|$ is the number of edges in~$\omega$,~$\ell(\omega)$ is the number of loops in~$\omega$ and
$Z_{\mathrm{loop}}(\calD,n,x)$ is a constant called the partition function, chosen so that ${\sf Loop}_{\calD,n,x}$ is a probability measure. 

We will consider that the origin~$0$ is a vertex of the hexagonal lattice and will always consider domains~$\calD$ containing~$0$. 
We say that the loop~$O(n)$ model with edge-weight~$x$ exhibits exponential decay of loop lengths if there exists~$c>0$ such that for any~$k\ge1$ and any domain~$\calD$, 
\begin{equation}\label{eq:exp-dec}
	{\sf Loop}_{\calD,n,x}[\mathsf R\ge k]\le\exp(-ck),
\end{equation}
where~$\mathsf R$ stands for the length of the biggest loop surrounding $0$.

According to physics predictions~\cite{Nie82,BloNie89}, the loop~$O(n)$ model exhibits macroscopic loops 
(in particular $\mathsf R$ is of the order of the radius of $\calD$) 
when~$n\in [1,2]$ and~$x\geq x_c(n) = \tfrac{1}{\sqrt{2+\sqrt{2-n}}}$. 
For all other values of~$n$ and~$x$, the model is expected to exhibit exponential decay. 
Moreover, it was conjectured (see e.g.~\cite[Section 5.6]{KagNie04}) that in the macroscopic-loops phase, the model has a conformally invariant scaling limit given by the Conformal Loop Ensemble (CLE) of parameter~$\kappa$, where:
\[
	\kappa=
	\begin{cases} 
		\frac{4\pi}{2\pi-\arccos(-n/2)}\in [\tfrac83, 4] &\text{ if }x=x_c(n),\\
		\tfrac{4\pi}{\arccos(-n/2)}\in [4, 8] &\text{ if }x>x_c(n).
	\end{cases}
\]

Our main result below is in agreement with the predicted phase diagram.

\begin{thm}\label{thm:exp_decay}
	For any~$n>1$, there exists $\eps(n) > 0$ such that the loop~$O(n)$ model exhibits exponential decay~\eqref{eq:exp-dec} 
	for all~$x < \frac{1}{\sqrt 3} + \eps(n)$.
\end{thm}

Prior to our work, the best known bound on the regime of exponential decay for~$n>1$ was~$x < \tfrac{1}{\sqrt{2+\sqrt{2}}}+ \varepsilon(n)$~\cite{Tag18}, where~$\sqrt{2+\sqrt{2}}$ is the connective constant of the hexagonal lattice computed in~\cite{DumSmi12}.  
Also, in~\cite{DumPelSam14} it was shown that when~$n$ is large enough the model exhibits exponential decay for any value of~$x>0$.
Apart from the improved result, our paper provides a method of relating (some forms of) monotonicity in $x$ and $n$; 
see Section~\ref{sec:open_questions} for more details. 

\begin{figure}
    \begin{center}
    \includegraphics[scale = 1.1]{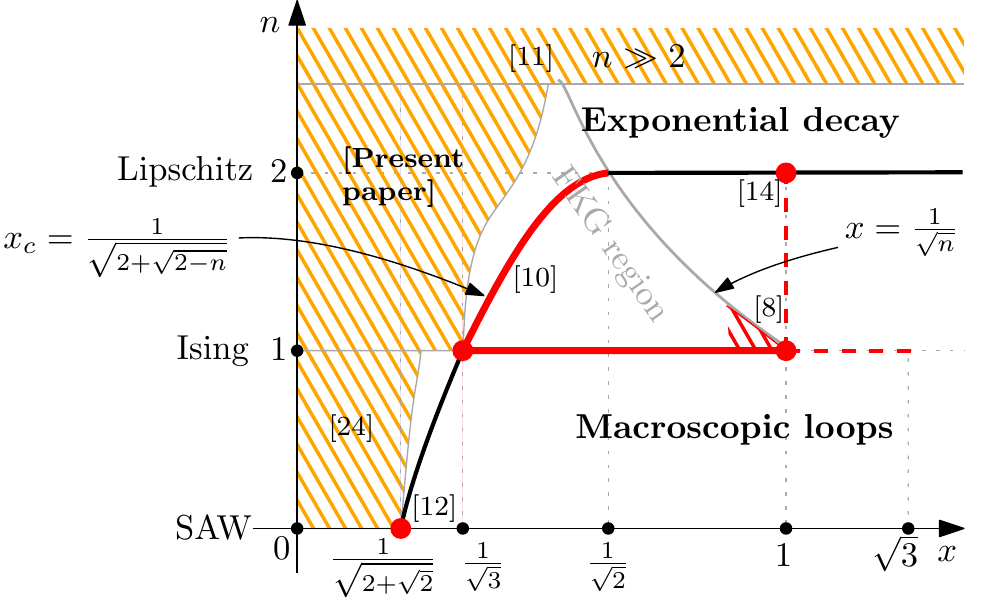}
    \cpsfrag{a}{\cite{Tag18}}
    \caption{The phase diagram of the loop $O(n)$ model. 
    It is expected that above and to the left of the curve $x_c(n)= \frac{1}{\sqrt{2+\sqrt{2-n}}}$ (in black) 
    the model exhibits exponential decay of loop lengths; 
    below and on the curve, it is expected to have macroscopic loops and converge to $\mathrm{CLE}(\kappa)$ in the scaling limit. 
    The convergence was established only at~$n = 1,x=\tfrac{1}{\sqrt{3}}$ (critical Ising model~\cite{Smi10,CheSmi12,BenHon16}) 
    and~$n=x=1$ (site percolation on~$\bbT$ at~$p_c=\tfrac{1}{2}$~\cite{Smi01, CamNew06}).
    Regions where the behaviour was confirmed by recent results are marked in orange (for exponential decay) and red (for macroscopic loops).
    The relevant references are also marked. }
    \end{center}
    \label{fig:phase_diagram}
\end{figure}

Existence of macroscopic loops was shown for~$n\in [1,2]$ and~$x=x_c(n) = \tfrac{1}{\sqrt{2+\sqrt{2-n}}}$ in~\cite{DumGlaPel17}, for~$n=2$ and 
$x=1$ in~\cite{GlaMan18}, and for~$n\in [1,1+\varepsilon]$ and~$x\in [1-\varepsilon, \tfrac{1}{\sqrt{n}}]$ in~\cite{CraGlaHarPel18}.
Additionally, for~$n=1$ and~$x\in [1,\sqrt{3}]$ (which corresponds to the antiferromagnetic Ising model) as well as for~$n\in[1,2]$ and~$x=1$, a partial result in the same direction was shown in~\cite{CraGlaHarPel18}.
Indeed, it was proved that in this range of parameters, at least one loop of length comparable to the size of the domain exists with positive probability (thus excluding the exponential decay).
All results appear on the phase diagram of Figure~\ref{fig:phase_diagram}.

We finish the introduction by providing a sketch of our proof. There are three main steps in it. Fix~$n>1$. 
First, inspired by Chayes--Machta~\cite{ChaMac98}, we develop the partition function in~$n = (n-1) + 1$, 
so that it takes the form of the loop~$O(1)$ model sampled on the vacant space of a weighted loop~$O(n-1)$ model. 
Second, we use that the loop~$O(1)$ model is the representation of the Ising model on the faces of~$\bbH$; the latter exhibiting exponential decay of correlations for all~$x<1/\sqrt{3}$. Via the FK-Ising representation, this statement may be extended when the Ising model is sampled in the random domain given by a loop~$O(n-1)$ configuration. 
At this stage we will have shown that the loop~$O(n)$ model exhibits exponential decay when~$x < 1/\sqrt 3$.
Finally, using enhancement techniques, we show that the presence of the loop~$O(n-1)$ configuration strictly increases the critical parameter of the Ising model, thus allowing to extend our result to all~$x<1/\sqrt{3} + \varepsilon(n)$.
\bigskip

\noindent {\bf Acknowledgements:}
The authors would like to thank Ron Peled for suggesting to develop in~$n = (n-1) + 1$ following Chayes and Machta. Our discussions with Hugo Duminil-Copin, Yinon Spinka and Marcelo Hilario were also very helpful.
We acknowledge the hospitality of IMPA (Rio de Janeiro), where this project started. 

The first author is supported by the Swiss NSF grant P300P2\_177848, and partially supported by the European Research Council starting grant 678520 (LocalOrder).
The second author is a member of the NCCR SwissMAP.

\section{The Ising connection}
	
In this section we formalise a well-known connection between the Ising model (and its FK-representation) and the loop~$O(1)$ model (see for instance \cite[Sec.~3.10.1]{FriVel17}).
It will be useful to work with inhomogeneous measures in both models. 

Fix a domain~$\calD = (V,E)$; we will omit it from notation when not necessary. 
Let~$\bfx = (x_e)_{E} \in [0,1]^{E}$ be a family of parameters. 
The loop~$O(1)$ measure with parameters~$\bfx$ is given by 
\begin{align*}
	{\sf Loop}_{\calD,1,\bfx}(\omega) =
	{\sf Loop}_{\bfx}(\omega) = \frac{1}{Z_{\mathrm{loop}}(\calD,1,\bfx)}\Big( \prod_{e \in \omega} x_e\Big) \cdot \ind_{\{\omega \text{ loop config.}\}}
	\qquad \text{ for all~$\omega \in \{0,1\}^{E}$}.
\end{align*}
The percolation measure~${\sf Perco}_{\bfx}$ of parameters~$\bfx$ consists of choosing the state of every edge independently, open with probability~$x_e$ for each edge~$e \in E$: 
\begin{align*}
	{\sf Perco}_{\bfx}(\omega) 
	= \Big(\prod_{e \in \omega} x_e\Big)\,\Big( \prod_{e \in E\setminus \omega}(1- x_e)\Big),
	\qquad \text{ for all~$\omega \in \{0,1\}^{E}$}.
\end{align*}
Finally, associate to the parameters~$\bfx$ the parameters~$\bfp = (p_e)_{E} \in [0,1]^{E}$ defined by 
\begin{align*}
	p_e = p(x_e) = \frac{2x_{e}}{1+x_e},\qquad \text{ for all~$e \in {E}$}.
\end{align*}
Define the FK-Ising measure on~$\calD$ by
\begin{align*}
	\Phi_{\bfx} (\omega) 
	= \frac{1}{Z_{FK}(\bfx)}\,\Big( \prod_{e \in \omega} p_e\Big)\,\Big( \prod_{e \in E\setminus \omega}(1- p_e)\Big)\, 2^{k(\omega)},
	\qquad \text{ for all~$\omega \in \{0,1\}^{E}$}.
\end{align*}
where~$k(\omega)$ is the number of connected components of~$\omega$ and~$Z_{FK}(\bfx)$ is a constant chosen so that~$\Phi_{\bfx}$ is a probability measure.

When~$\bfx$ is constant equal to some~$x \in [0,1]$, write~$x$ instead of~$\bfx$.
For~$D \subset E$, write~$\Phi_{D, x}$ and~${\sf Loop}_{D,1,x}$
for the FK-Ising and loop~$O(1)$ measures, respectively, on~$\calD$ with inhomogeneous weights~$(x \ind_{\{e \in D\}})_{e\in E}$
(where $\ind$ stands for the indicator function).
These are simply the measures~$\Phi_{\calD, x}$ and~${\sf Loop}_{\calD,1,x}$ conditioned on~$\omega\cap D^c = \emptyset$. 

\begin{prop}\label{prop:sub_FK}
	Fix~$\bfx = (x_e)_{E} \in [0,1]^{E}$ and let  
	$\omega, \pi \in \{0,1\}^{E}$ be two independent configurations chosen according to~${\sf Loop}_{\bfx}$ and~${\sf Perco}_{\bfx}$, respectively. 
	Then the configuration~$\omega \vee \pi$ defined by~$(\omega \vee \pi)(e) = \max\{\omega(e),\pi(e)\}$ has law~$\Phi_{\bfx}$.
	In particular 
	\begin{align}\label{eq:sub_FK}
		{\sf Loop}_{\bfx} \leq_{\text{st}}\Phi_{\bfx}.
	\end{align}
\end{prop}
	
We give a short proof below. The reader familiar with the Ising model may consult the diagram of Figure~\ref{fig:coupling} for a more intricate but more natural proof. 

\begin{figure}
    \begin{center}
    \includegraphics[width=1\textwidth]{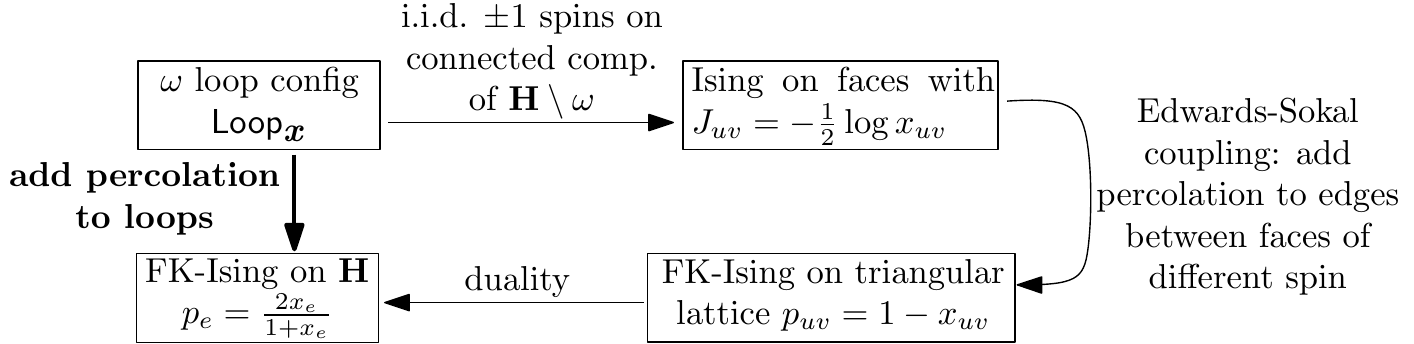}
    \caption{The coupling of Proposition~\ref{prop:sub_FK} via the spin-Ising representation.}
    \label{fig:coupling}
    \end{center}
\end{figure}

\begin{proof}
	Write~${\sf Loop}_{\bfx} \otimes {\sf Perco}_{\bfx}$ for the measure sampling~$\omega$ and~$\pi$ independently. 
	Fix~$\eta \in \{0,1\}^{E}$ and let us calculate 
	\begin{align}\label{eq:coupling1}\nonumber
		{\sf Loop}_{\bfx} \otimes {\sf Perco}_{\bfx}(\omega \vee \pi = \eta)
		& = \sum_{\substack{\omega \subset \eta \\\omega \text{ loop config}}} 
			{\sf Loop}_{\calD,\bfx}(\omega)\cdot  {\sf Perco}_{\calD\setminus \omega, \bfx}(\eta \setminus \omega)\\
		& = \sum_{\substack{\omega \subset \eta \\\omega \text{ loop config}}} 
		\frac{1}{Z_{\mathrm{loop}}(\calD,1,\bfx)}
		\,\Big(\prod_{e \in \omega} x_e\Big)\,\Big(  \prod_{e \in\eta \setminus \omega } x_e\Big)\,\Big( \prod_{e \in E \setminus \eta } (1 - x_e) \Big) \nonumber\\
		& =\frac{1}{Z_{\mathrm{loop}}(\calD,1,\bfx)}\,\Big( \prod_{e \in\eta} x_e\Big)\,\Big( \prod_{e \in E \setminus \eta } (1 - x_e) \Big)
		\sum_{\substack{\omega \subset \eta \\\omega \text{ loop config}}} 1.
	\end{align}
	
	Next we estimate the number of loop configurations~$\omega$ contained in~$\eta$. 
	Consider~$\eta$ as a graph embedded in the plane 
	and let~$F(\eta)$ be the set of connected components of~$\bbR^2 \setminus \eta$; these are the faces of~$\eta$. 
	The set of loop configurations~$\omega$ contained in~$\eta$ is in bijection with the set of assignments of spins~$\pm 1$ 
	to the faces of~$\eta$, with the only constraint that the infinite face has spin~$+1$. 
	Indeed, given a loop configuration~$\omega \subset \eta$, assign spin~$-1$ to the faces of~$\eta$ surrounded by an odd number of loops of~$\omega$, and~$+1$ to all others. 
	The inverse map is obtained by considering the edges separating faces of distinct spin. 
	
	The Euler formula applied to the graph~$\eta$ reads~$|V| - |\eta| + |F(\eta)| = 1 + k(\eta)$.
	Hence, the number of loop configurations contained in~$\eta$ is 
	\begin{align*}
	\sum_{\substack{\omega \subset \eta \\\omega \text{ loop config}}} 1 = 2^{F(\eta) -1}  =  2^{k(\eta) + |\eta| - |V|}.
	\end{align*}
	Inserting this in \eqref{eq:coupling1}, we find
	\begin{align*}
		{\sf Loop}_{\bfx} \otimes {\sf Perco}_{\bfx}\big(\omega \vee \pi = \eta\big)
		& =\frac{2^{-|V|}}{Z_{\mathrm{loop}}(\calD,1,\bfx)}\,\Big( \prod_{e \in\eta} 2 x_e\Big)\,\Big( \prod_{e \in E \setminus \eta} (1 - x_e)\Big)\,  2^{k(\eta)}\\
		& =\frac{2^{-|V|}(1+x_e)^{|E|}}{Z_{\mathrm{loop}}(\calD,1,\bfx)} 
		\,\Big(\prod_{e \in\eta} \tfrac{2x_e}{1+x_e}\Big)\,\Big( \prod_{e \in E \setminus \eta} (1 - \tfrac{2x_e}{1+x_e})\Big) \,2^{k(\eta)}.
	\end{align*}
	Since~${\sf Loop}_{\bfx} \otimes {\sf Perco}_{\bfx}$ is a probability measure, we deduce that it is equal to~$\Phi_{\bfx}$ and that the normalising constants are equal, namely
	\begin{align}\label{eq:Z}
		\frac{Z_{\mathrm{loop}}(\calD,1,\bfx)}{2^{-|V|}(1+x_e)^{|E|}} = {Z_{\text{FK}}(\calD,\bfx)}.
	\end{align}
\end{proof}

While the loop model has no apparent monotonicity, the FK-Ising model does. 
This will be of particular importance. 

\begin{prop}[Thm.~3.21 \cite{Gri06}]\label{prop:monotonicity}
	Let~$\bfx = (x_e)_{e \in E} \in [0,1]^E$ and~$\tilde\bfx = (\tilde x_e)_{e \in E} \in [0,1]^E$ be two sets of parameters 
	with~$x_e \leq \tilde x_e$ for all~$e \in E$. Then~$\Phi_{\bfx}  \leq_{\text{st}} \Phi_{\tilde\bfx}$ in the sense that
	\begin{align*}
		\Phi_{\bfx}(A)  \leq \Phi_{\tilde\bfx}(A) \qquad \text{ for any increasing event~$A$.}
	\end{align*}
\end{prop}
The version above is slightly different from \cite[Thm~3.21]{Gri06}, as it deals with inhomogeneous measures;
adapting the proof is straightforward. 

Finally, it is well known that the FK-Ising model on the hexagonal lattice exhibits a sharp phase transition at~$p_c = \frac2{\sqrt3 + 1}$~--- the critical point for the Ising model was computed by Onsager~\cite{Ons44} 
(see~\cite{BefDum12} for the explicit formula on the triangular lattice), 
the sharpness of the phase transition was shown in~\cite{AizBarFer87}.
For~$p = p(x)$ strictly below $p_c$, which is to say~$x < \frac{1}{\sqrt 3}$, the model exhibits exponential decay of cluster volumes. 
Indeed, this may be easily deduced using~\cite[Thm.~5.86]{Gri06}.

\begin{thm}\label{thm:FK_subcrit}
	For~$x < \frac{1}{\sqrt 3}$ there exist~$c= c(x) > 0$ and~$C>0$ such that, for any domain~$\calD$,
	\begin{align*}
		\Phi_{\calD,x}(|\calC_0|\geq k) \leq C\, e^{-c\, k},
	\end{align*}
	where~$\calC_0$ denotes the cluster containing~$0$ and $|\calC_0|$ its number of vertices.
\end{thm}

\section{n = (n-1) + 1}

Fix a domain~$\calD = (V,E)$ and a value~$n > 1$. 
Choose~$\omega$ according to~${\sf Loop}_{\calD,n,x}$. 
Colour each loop of~$\omega$ in blue with probability~$1- \frac1n$ and red with probability~$\frac1n$.
Let~$\omega_b$ and~$\omega_r$ be the configurations formed only of the blue and red loops, respectively;
extend~${\sf Loop}_{\calD,n,x}$ to incorporate this additional randomness. 

\begin{prop}\label{prop:marginals}
	For any two non-intersecting loop configurations $\omega_b$ and $\omega_r$, we have
	\begin{align*}
		{\sf Loop}_{\calD,n,x} (\omega_r \,|\, \omega_b) & = {\sf Loop}_{\calD\setminus \omega_b,1,x}(\omega_r)		 \qquad\qquad\qquad \text{and}\medskip\\
		{\sf Loop}_{\calD,n,x} (\omega_b) & =\frac{{Z_{\mathrm{loop}}(\calD\setminus \omega_b,1,x)}}{Z_{\mathrm{loop}}(\calD,n,x)}(n-1)^{\ell(\omega_b)} x^{|\omega_b|}.
	\end{align*}
\end{prop}

\begin{proof}
	For two non-intersecting loop configurations~$\omega_b$ and~$\omega_r$, if we write~$\omega = \omega_b\vee\omega_r$, we have
	\begin{align*}
    	{\sf Loop}_{\calD,n,x}(\omega_b,\omega_r)
    	&= \big(\tfrac{n-1}{n} \big)^{\ell(\omega_b)}\big(\tfrac{1}n \big)^{\ell(\omega_r)}{\sf Loop}_{\calD,n,x}(\omega)\\
		&= \frac{1}{Z_{\mathrm{loop}}(\calD,n,x)} (n-1)^{\ell(\omega_b)} x^{|\omega_b| + |\omega_r|}\\
		&= \frac{{Z_{\mathrm{loop}}(\calD\setminus \omega_b,1,x)}}{Z_{\mathrm{loop}}(\calD,n,x)}(n-1)^{\ell(\omega_b)} x^{|\omega_b|} \cdot
		\frac{x^{|\omega_r|}}{Z_{\mathrm{loop}}(\calD\setminus \omega_b,1,x)}.
	\end{align*}
	Notice that~$\omega_r$ only appears in the last fraction. 
	Moreover, if we sum this fraction over all loop configurations~$\omega_r$ not intersecting~$\omega_b$, we obtain~$1$. 
	This proves both assertions of the proposition. 
\end{proof}

Recall that for a percolation configuration, $\calC_0$ denotes the connected component containing~$0$. 
If $\omega$ is a loop configuration, then~$\calC_0(\omega)$ is simply the loop in~$\omega$ that passes through~$0$ 
(with~$\calC_0(\omega):=\{0\}$ if no such loop exists).

\begin{cor}\label{cor:thm-easy-part}
	Let~$n\geq 1$ and~$x<1/\sqrt{3}$. Then~${\sf Loop}_{\calD,n,x}$ exhibits exponential decay.
\end{cor}

\begin{proof}
	For any domain~$\calD$ and~$k\geq 1$ we have
	\begin{align*}
		\tfrac{1}{n} {\sf Loop}_{\calD,n,x} (|\calC_0(\omega)|\geq k)
		&= {\sf Loop}_{\calD,n,x} (|\calC_0(\omega_r)|\geq k)\\
		&\leq {\sf Loop}_{\calD,n,x}\big[\Phi_{\calD\setminus\omega_b, x}( |\calC_0|\geq k) \big] \qquad  &\text{ by Prop.~\ref{prop:sub_FK} }\\
		&\leq \Phi_{\calD, x}(  |\calC_0|\geq k) \qquad & \text{ by Prop.~\ref{prop:monotonicity} }\\
		&\leq C e^{ - c\,k} &  \text{ by Thm.~\ref{thm:FK_subcrit},}
	\end{align*}
	where~$c = c(x) >0$ and~$C> 0$ are given by Theorem~\ref{thm:FK_subcrit}.
	Thus, the length of the loop of~$0$ has exponential tail, uniformly in the domain~$\calD$. 
	In particular, if~$\calD$ is fixed, the above bound also applies to any translates of~$\calD$, 
	hence to the loop of any given point in~$\calD$. 
	
	Let $v_0,v_1,v_2\dots$ be the vertices of $\calD$ on the horizontal line to the right of $0$, 
	ordered from left to right, starting with $v_0 = 0$. 
	If ${\sf R} \geq k$, then the largest loop surrounding $0$ 
	either passes through one of the points $v_0,\dots, v_{k-1}$ and has length at least $k$, 
	or it passes through some $v_j$ with $j \geq k$, and has length at least $j$, so as to manage to surround $0$. 
	Thus, using the bound derived above, we find
	\begin{align*}
		{\sf Loop}_{\calD,n,x} ({\sf R} \geq k)
		\leq n\, \big[C\, k e^{ - c\, k} + \sum_{j\geq k } C\, e^{ - c\,j}\big]
		\leq C'e^{ - c'\, k},
	\end{align*}
	for some altered constants $c'>0$ and $C'$ that depend on $c,C$ and $n$ but not on $k$. 
\end{proof}

\section{A little extra juice: enhancement}\label{sec:enhancement}

Fix some domain~$\calD = (V,E)$ for the whole of this section.
Let~$\omega_b$ be a blue loop configuration.
Associate to it the spin configuration~$\sigma_b \in \{-1,+1\}^{F_\calD}$ obtained by awarding spins~$-1$ to all faces of~$\calD$ that are surrounded by an odd number of loops, and spins~$+1$ to all other faces. 
Write~$\sfD_+ = \sfD_+(\sigma_b)$ (and~$\sfD_- = \sfD_-(\sigma_b)$, respectively) 
for the set of edges of~$\calD$ that have~$\sigma_b$-spin~$+1$ (and~$-1$, respectively) on both sides. 
All faces outside of~$\calD$ are considered to have spin~$+1$ in this definition.
Equivalently,~$\sfD_-$ is the set of edges of~$\calD\setminus \omega_b$ surrounded by an odd number of loops of~$\omega_b$ 
and~$\sfD_+ = \calD\setminus (\omega_b\cup \sfD_-)$.
Both~$\sfD_+$ and~$\sfD_-$ will also be regarded as spanning subgraphs of~$\calD$ with edge-sets~$\sfD_+$ and~$\sfD_-$, respectively. 

Since no edge of~$\sfD_+$ is adjacent to any edge of~$\sfD_-$, 
a sample of the loop~$O(1)$ measure~${\sf Loop}_{\calD\setminus\omega_b,1, x}$
may be obtained by the superposition of two independent samples from~${\sf Loop}_{\sfD_+,1, x}$ and~${\sf Loop}_{\sfD_-,1, x}$, respectively. 
In particular, using \eqref{eq:sub_FK},
\begin{align}\label{eq:D+D-}\nonumber
	{\sf Loop}_{\calD\setminus\omega_b,1, x} (|\calC_0(\omega_r)|\geq k)
	&= {\sf Loop}_{\sfD_+,1,x} (|\calC_0(\omega_r)|\geq k) +  {\sf Loop}_{\sfD_-,1,x} (|\calC_0(\omega_r)|\geq k)\\
	&\leq \Phi_{\sfD_+,x} (|\calC_0|\geq k) + \Phi_{\sfD_-,x} (|\calC_0|\geq k). 
\end{align}
Actually, depending on~$\omega_b$, at most one of the terms on the RHS above is non-zero. 
We nevertheless keep both terms as we will later average on~$\omega_b$.
The two following lemmas will be helpful in proving Theorem~\ref{thm:exp_decay}.

\begin{lem}\label{lem:domination_D}
	Let~$x < 1$ and set
	\begin{align}\label{eq:alpha}
		\alpha = \Big[ \frac{ \max\{(n-1)^{2}, (n-1)^{-2}\}}{(x/2)^6 + \max\{(n-1)^{2}, (n-1)^{-2}\}}\Big]^{1/6} < 1. 
	\end{align}
	If~$\omega_b$ has the law of the blue loop configuration of~${\sf Loop}_{\calD,n,x}$, 
	then both laws of~$\sfD_+$ and~$\sfD_-$ are stochastically dominated by~${\sf Perco}_{\alpha}$.
\end{lem}

\begin{lem}\label{lem:domination_omega}
	Fix $x \in (0,1)$ and $\alpha<1$. Let $\tilde x < x$ be such that 
	\begin{align}\label{eq:tildex}
		\frac{\tilde x}{1-\tilde x}= 
		\frac{x}{1-x}\cdot \Big(1+ \frac{1+x}{2(1-x)}\cdot \frac{1-\alpha}\alpha\Big)^{-1}.
	\end{align}
	Write~${\sf Perco}_{\alpha}(\Phi_{\sfD,x}(.))$ for the law of~$\eta$ chosen using the following two step procedure:
	choose~$\sfD$ according to~${\sf Perco}_{\alpha}$, then choose~$\eta$ according to~$\Phi_{\sfD,x}$. 
	Then 	
	\begin{align*}
		{\sf Perco}_{\alpha}(\Phi_{\sfD,x}(.)) \leq_{\text{st}} \Phi_{\calD,\tilde x}.
	\end{align*}
\end{lem}

Before proving the two lemmas above, let us show that they imply the main result. 

\begin{proof}[of Theorem \ref{thm:exp_decay}]
	Fix $n > 1$. 
	An elementary computation proves the existence of some $\eps = \eps (n) >0$
	such that, if $x < \frac{1}{\sqrt 3} + \eps(n)$ and $\alpha$ and $\tilde x$ are defined in terms of~$n$ and~$x$ via \eqref{eq:alpha} and \eqref{eq:tildex}, respectively, 
	then $\tilde x < \frac{1}{\sqrt 3}$~
	\footnote{When $n \searrow 1$, we have $\eps(n) \sim C (n-1)^2$, where $C = \frac{1+ \sqrt 3}{12^4\,  \sqrt3}$.}.

	Fix~$x < \frac{1}{\sqrt 3} + \eps(n)$ along with the resulting values $\alpha$ and $\tilde x < \frac{1}{\sqrt 3}$.
	Then, for any domain~$\calD$ and~$k\geq 1$ we have
	\begin{align*}
		\tfrac{1}{n} {\sf Loop}_{\calD,n,x} (|\calC_0(\omega)|\geq k)
		&= {\sf Loop}_{\calD,n,x} (|\calC_0(\omega_r)|\geq k)\\
		&\leq {\sf Loop}_{\calD,n,x}\big[\Phi_{\sfD_+, x}( |\calC_0|\geq k)
		+ \Phi_{\sfD_-, x}(  |\calC_0|\geq k) \big] 	   &\text{ by \eqref{eq:D+D-}}\\
		&\leq 2\,{\sf Perco}_{\alpha}\big[\Phi_{\sfD, x}(  |\calC_0|\geq k) \big]   &\text{ by Lemma~\ref{lem:domination_D}}\\
		&\leq 2\, \Phi_{\calD, \tilde x}( |\calC_0|\geq k)    &\text{ by Lemma~\ref{lem:domination_omega}}\\
		&\leq 2\,C\, e^{ - c\,k} &  \text{ by Thm.~\ref{thm:FK_subcrit}.}
	\end{align*}
	In the third line, we have used Lemma~\ref{lem:domination_D} and the stochastic monotonicity of~$\Phi$ in terms of the domain. 
	Indeed,  Lemma~\ref{lem:domination_D} implies that ~${\sf Loop}_{\calD,n,x}$ and~${\sf Perco}_{\alpha}$ may be coupled so that the sample~$\sfD_+$ obtained from the former is included in the sample~$\sfD$ obtained from the latter. Thus~$\Phi_{\sfD_+,x} \leq \Phi_{\sfD, x}$. The same applies separately for~$\sfD_-$. 
	
	To conclude \eqref{eq:exp-dec}, continue in the same way as in the proof of Corollary \ref{cor:thm-easy-part}.
\end{proof}

The following computation will be useful for the proofs of both lemmas. 
Let~$D \subset E$ and~$e \in E\setminus D$. 
We will also regard~$D$ as a spanning subgraph of~$\calD$ with edge-set~$D$. 
Recall that~$Z_{FK}(D,x)$ is the partition function of the FK-Ising measure~$\Phi_{D,x}$ on~$D$. 
Then
\begin{align}
    Z_{FK}(D,x) 
    &= \sum_{\eta \subset D} p^{|\eta|}(1-p)^{|D| - |\eta|}2^{k(\eta)}\nonumber\\
    &= \sum_{\eta \subset D} 
    p^{|\eta|}(1-p)^{|D\cup\{e\}| - |\eta|}\, 2^{k(\eta)} + p^{|\eta\cup\{e\}|}(1-p)^{|D| - |\eta|}\,2^{k(\eta)} \nonumber\\
    &\geq \sum_{\eta \subset D\cup\{e\}} p^{|\eta|}(1-p)^{|D\cup\{e\}| - |\eta|}2^{k(\eta)}
	= Z_{FK}(D\cup\{e\},x),\label{eq:Z1}
\end{align}
since~$k(\eta) \geq k(\eta \cup \{e\})$. Conversely,~$k(\eta) \leq k(\eta \cup \{e\}) +1$, which implies 
\begin{align}\label{eq:Z2}
	Z_{FK}(D,x) \leq 2\, Z_{FK}(D \cup\{e\},x).
\end{align}

\begin{proof}[of Lemma~\ref{lem:domination_D}]
	For~$\beta \in (0,1)$ let~$P_\beta$ be the Bernoulli percolation on the faces of~$\calD$ of parameter~$\beta$:
	\begin{align*}
		P_{\beta}(\sigma) 
		= \beta^{\#\{u\,:\, \sigma(u) = +1\}} (1-\beta)^{\#\{u\,:\, \sigma(u) = -1\}} \qquad \text{ for all~$\sigma \in \{-1,+1\}^{F_\calD}$}.
	\end{align*}
	To start, we will prove that the law induced on~$\sigma_b$ by~${\sf Loop}_{\calD,n,x}$ is dominated by~$P_\beta$ for some~$\beta$ sufficiently close to~$1$.
	Both measures are positive, and Holley's inequality \cite{Hol74} states that the stochastic ordering is implied by 
	\begin{align*}
		\frac{{\sf Loop}_{\calD,n,x}(\sigma_b = \varsigma_1)}{{\sf Loop}_{\calD,n,x}(\sigma_b = \varsigma_1 \wedge \varsigma_2)}
		\leq 
		\frac{P_{\beta}(\varsigma_1 \vee \varsigma_2)}{P_{\beta}(\varsigma_2)} = \Big(\frac{\beta}{1 - \beta}\Big)^{\#\{u \,:\, \varsigma_1(u) =+1,\, \varsigma_2(u) = -1\}}
		\quad \text{ for all~$\varsigma_1,\varsigma_2\in \{\pm1\}^{F_\calD}$}.
	\end{align*}
	The RHS above only depends on the number of faces of spin $+$ in $\varsigma_1$ and spin $-$ in $\varsigma_2$. 
	It is then elementary to check that the general inequality above is implied by the restricted case where~$\varsigma_1$ differs at exactly one face~$u$ from~$\varsigma_2$, 
	and~$\varsigma_1(u) =+1$ but~$\varsigma_2(u) = -1$.
	
	Fix two such configurations~$\varsigma_1$,~$\varsigma_2$; write~$\omega_1$ and~$\omega_2$ for their associated loop configurations.
	Then, by Lemma~\ref{prop:marginals}, 
	\begin{align*}
		\frac{{\sf Loop}_{\calD,n,x}(\sigma_b = \varsigma_1)}{{\sf Loop}_{\calD,n,x}(\sigma_b = \varsigma_2)}
		&= \frac{{Z_{\mathrm{loop}}(\calD\setminus \omega_1,1,x)}}{{Z_{\mathrm{loop}}(\calD\setminus \omega_2,1,x)}}\,
		(n-1)^{\ell(\omega_1) - \ell(\omega_2)} x^{|\omega_1| - |\omega_2|}.
	\end{align*}
	Since~$\varsigma_1$ and~$\varsigma_2$ only differ by one face, 
	$\omega_1$ and~$\omega_2$ differ only in the states of the edges surrounding that face.
	In particular~$||\omega_1| - |\omega_2||\leq 6$ and~$|\ell(\omega_1) - \ell(\omega_2)| \leq 2$.
	Finally, using \eqref{eq:Z1} and \eqref{eq:Z2}, we find 
	\begin{align*}
		 \frac{{Z_{\mathrm{loop}}(\calD\setminus \omega_1,1,x)}}{{Z_{\mathrm{loop}}(\calD\setminus \omega_2,1,x)}}	
		 \leq 
		 \frac{{Z_{\mathrm{loop}}(\calD\setminus (\omega_1 \wedge \omega_2),1,x)}}{{Z_{\mathrm{loop}}(\calD\setminus \omega_2,1,x)}}
		 \leq 2^{|\omega_2| - |\omega_1 \wedge \omega_2|} \leq 2^6.	
	\end{align*}
	In conclusion 
	\begin{align*}
		\frac{{\sf Loop}_{\calD,n,x}(\sigma_b = \varsigma_1)}{{\sf Loop}_{\calD,n,x}(\sigma_b = \varsigma_2)}
		\leq \Big(\frac2x\Big)^6\cdot \max\{(n-1)^{2}, (n-1)^{-2}\}.
	\end{align*}
	Then, if we set 
	\begin{align*}
		\beta = \frac{\big(\tfrac2x\big)^6\cdot \max\{(n-1)^{2}, (n-1)^{-2}\}}{1 +  \big(\tfrac2x\big)^6\cdot \max\{(n-1)^{2}, (n-1)^{-2}\}}, 
	\end{align*}
	we indeed obtain the desired domination of~$\sigma_b$ by~$P_\beta$
	\footnote{This domination is of special interest as~$n \searrow 1$ and for~$x \geq 1/\sqrt 3$. 
	Then we may simplify the value of~$\beta$ as 
	$\beta = \frac{(2\sqrt3)^6}{ (n-1)^{2} + (2\sqrt3)^6} \sim 1 - \tfrac1{(2\sqrt3)^{6}} (n-1)^{2}~$.}. 
	The same proof shows that~$-\sigma_b$ is also dominated by~$P_\beta$. 
	
	\smallskip
	
	Next, le us prove the domination of~$\sfD_+$ by a percolation measure. 
	Set~$\alpha  = \beta^{1/6}$~\footnote{As ~$n \searrow 1$ and~$x \geq 1/\sqrt 3$, 
	we may assume that~$\alpha \sim 1 - \tfrac{1}{6\, (2\sqrt3)^{6}} (n-1)^{2}$.}.
	Let~$\eta_L$ and~$\eta_R$ be two percolation configurations chosen independently according to~${\sf Perco}_\alpha$.
	Also choose an orientation for every edge of~$E$; for boundary edges, 
	orient them such that the face of~$\calD$ adjacent to them is on their left. 
	
	Define~$\tilde\sigma \in \{\pm1\}^{F_\calD}$ as follows.
	Consider some face~$u$. 
	For an edge~$e$ adjacent to~$u$,~$u$ is either on the left of~$e$ or on its right, according to the orientation chosen for~$e$. 
	If it is on the left, retain the number~$\eta_L(e)$, otherwise retain~$\eta_R(e)$. 
	Consider that~$u$ has spin~$+1$ under~$\tilde\sigma$ if and only if all the six numbers retained above are~$1$. 
	Formally, for each~$u \in F_\calD$, set~$\tilde\sigma (u) = +1$ if and only if
	\begin{align*}
		\prod_{e \text{ adjacent to } u} 
		\big(\eta_L(e) \ind_{\{u \text{ is left of } e\}} + \eta_R(e)\ind_{\{u \text{ is right of } e\}}\big) =1.
	\end{align*}

	As a consequence, for an edge~$e$ in the interior of~$\calD$ to be in~$\sfD_+(\tilde\sigma)$, 
 	the faces on either side of~$e$ need to have~$\tilde\sigma$-spin~$+1$, 
	hence~$\eta_L(e) = \eta_R(e) = 1$ is required.  
	For boundary edges~$e$ to be in~$\sfD_+(\tilde\sigma)$, only the restriction~$\eta_L(e) = 1$ remains. 
	In conclusion~$\eta_L \geq \sfD_+(\tilde\sigma)$. 
	
	Let us analyse the law of~$\tilde\sigma$. 
	Each value~$\eta_L(e)$ and~$\eta_R(e)$ appears in the definition of one~$\tilde\sigma(u)$. 
	As a consequence, the variables~$\big(\tilde\sigma(u)\big)_{u\in F}$ are independent. 
	Moreover,~$\tilde\sigma(u) = 1$ if and only if all the six edges around~$e$ are open
	in one particular configuration~$\eta_L$ or~$\eta_R$, which occurs with probability~$\alpha^6 = \beta$. 
	As a consequence~$\tilde\sigma$ has law~$P_\beta$. 
	
	By the previously proved domination,~${\sf Loop}_{\calD,n,x}$ may be coupled with~$P_\beta$ so that~$\tilde\sigma \geq \sigma_b$. 
	If this is the case, we have
    \begin{align*}
		\eta_L \geq \sfD_+(\tilde\sigma) \geq \sfD_+(\sigma_b).
    \end{align*}
	Thus,~$\eta_L$ indeed dominates~$\sfD_+(\sigma_b)$, as required.
	 
	The same proof shows that~$\eta_L$ dominates~$\sfD_-(\sigma_b)$. 
	For clarity, we mention that this does not imply that~$\eta_L$ dominates~$\sfD_+(\sigma_b)$ and~$\sfD_-(\sigma_b)$ simultaneously, 
	which would translate to~$\eta_L$ dominating~$\calD\setminus \omega_b$. 
\end{proof}

\begin{proof}[of Lemma~\ref{lem:domination_omega}]
	Fix $x$, $\alpha$ and $\tilde x$ as in the Lemma. 
	The statement of Holley's inequality applied to our case may easily be reduced to
	\begin{align}\label{eq:Holley2}
		\frac{{\sf Perco}_{\alpha}[\Phi_{\sfD,x}(\eta \cup\{e\})]}{{\sf Perco}_{\alpha}[\Phi_{\sfD,x}(\eta)]} 
		\leq \frac{\Phi_{\calD,\tilde{x}}(\tilde{\eta} \cup\{e\})}{\Phi_{\calD,\tilde{x}}(\tilde{\eta})}
		\qquad \text{ for all~$\eta \leq \tilde{\eta}$ and~$e \notin \tilde \eta$}.	
	\end{align}
	Fix $\eta, \tilde \eta$ and $e = (uv)$ as above. 
	For ~$D\subset E$ with~$e \in D$, a standard computation yields
	\begin{align*}
		\varphi_x(e|\eta) 
		:= \frac{\Phi_{D,x}(\eta \cup\{e\})}
		{\Phi_{D,x}(\eta)}
		= \begin{cases}
			\frac{2x}{1-x} \quad & \text{ if~$u \xlra{\eta} v$ and} \\
			\frac{x}{1-x} \quad & \text{ otherwise}.
		\end{cases} 
	\end{align*}
	The same quantity may be defined for~$\tilde x$ instead of~$x$ and~$\tilde \eta$ instead of~$\eta$;
	it is increasing in both~$\eta$ and~$x$. 
	Moreover~$\varphi_x(e|\eta)$ does not depend on~$D$, as long as~$e \in D$ and~$\eta \subset D$.
	If the first condition fails, then the numerator is~$0$; if the second fails then the denominator is null and the ratio is not defined. 
	
	Let us perform a helpful computation before proving \eqref{eq:Holley2}.
	Fix~$D$ with $e \in D$. By~\eqref{eq:Z2},  
	\begin{align*}
		\frac{\Phi_{D\setminus\{e\},x}(\eta)}{\Phi_{D,x}(\eta)}
		= \frac{Z_{FK}(D,x)}{Z_{FK}(D\setminus\{e\},x)}\cdot\frac{1+x}{1-x}
		\geq\frac{1+x}{2(1-x)}.
	\end{align*}
	The factor~$\big(\frac{1-x}{1+x}\big)^{-1}$ comes from the fact that the weights of~$\eta$ under 
	$\Phi_{D\setminus\{e\},x}$ and~$\Phi_{D,x}$ differ by the contribution of the closed edge~$e$. 
	If follows that
	\begin{align*}
		\Big(1+ \frac{1+x}{2(1-x)}\cdot \frac{1-\alpha}\alpha\Big)\,\Phi_{D,x}(\eta)
		\leq \Phi_{D,x}(\eta) +\frac{1-\alpha}{\alpha}\, \Phi_{D\setminus\{e\},x}(\eta).
	\end{align*}
	The choice of~$\tilde x$ is such that
	\begin{align*}
		\frac{\varphi_x(e|\eta)}{ \varphi_{\tilde x}(e|\eta)}
		= \frac{x}{1- x}\cdot \frac{1-\tilde x}{\tilde x}
		= 1+ \frac{1+x}{2(1-x)}\cdot \frac{1-\alpha}\alpha.
	\end{align*}
	Using the last two displayed equations, we find
	\begin{align*}
		&{\sf Perco}_{\alpha}[\Phi_{\sfD,x}(\eta \cup\{e\})]
		= \sum_{D \subset E} \alpha^{|D|}(1 - \alpha)^{|E| - |D|} 
		\, \Phi_{D,x}(\eta \cup \{e\})\\
		&\qquad= \sum_{\substack{D \subset E \\ \text{with }e\in D}} \alpha^{|D|}(1 - \alpha)^{|E| - |D|} 
		\, \varphi_x(e|\eta)\, \Phi_{D,x}(\eta)\\
		&\qquad\leq 
		\big(1+ \tfrac{1+x}{2(1-x)}\cdot \tfrac{1-\alpha}\alpha\big)^{-1} \varphi_x(e|\eta)
		\sum_{\substack{D \subset E \\ \text{with }e\in D}} \alpha^{|D|}(1 - \alpha)^{|E| - |D|}
		\,\Big[\Phi_{D,x}(\eta) +\tfrac{ 1-\alpha}\alpha \Phi_{D\setminus\{e\},x}(\eta)\Big]\\
		&\qquad= \varphi_{\tilde x}(e|\tilde{\eta})
		\sum_{D \subset E} \alpha^{|D|}(1 - \alpha)^{|E| - |D|} \Phi_{D,x}(\eta)	\\
		&\qquad= \varphi_{\tilde x}(e|\tilde{\eta}) \,
		{\sf Perco}_{\alpha}[\Phi_{\sfD,x}(\eta)].
	\end{align*}
	Divide by ${\sf Perco}_{\alpha}[\Phi_{\sfD,x}(\eta)]$ and recall the definition of $\varphi_{\tilde x}(e|\tilde{\eta})$ to obtain~\eqref{eq:Holley2}.
\end{proof}

\section{Open questions / perspectives}\label{sec:open_questions}

Our main theorem shows that if $x$ is such that the model with parameters $x$ and $n = 1$ exhibits exponential decay, 
then so do all models with the same parameter $x$ and $n \geq 1$. 
A natural generalisation of this is the following.

\begin{question}\label{question2}
	Show that if $x$ and $n$ are such that the loop $O(n)$ model exhibits exponential decay,
	then so do all models with parameters $x$ and $\tilde n$ for any $\tilde n \geq n$. 
\end{question}

A positive answer to the above would show that the critical point $x_c(n)$ (assuming it exists) is increasing in $n$. 
The same technique as in Section~\ref{sec:enhancement} may even prove that it is strictly increasing.
Moreover, it was recently shown in~\cite{DumGlaPel17} that, in the regime~$n\geq 1$ and $x\leq \tfrac{1}{\sqrt{n}}$, 
the loop~$O(n)$ model satisfies the following dichotomy: either it exhibits macroscopic loops or exponential decay. 
In addition, for~$n\in [1,2]$ and~$x=\tfrac{1}{\sqrt{2+\sqrt{2-n}}}$ the loop~$O(n)$ model was shown to exhibit macroscopic loops. 
Thus, assuming Question~\ref{question2}, 
we deduce that the loop $O(n)$ model with $ n \leq 2$ and $x \in [\tfrac{1}{\sqrt{2+\sqrt{2-n}}}, \tfrac{1}{\sqrt{2}}]$ exhibits macroscopic loops.\smallskip

Let us now describe a possible approach to Question~\ref{question2}.
The strategy of our proof of Theorem~\ref{thm:exp_decay} was based on the following observation. 
The loop $O(1)$ model, or rather its associated FK-Ising model, has a certain monotonicity in $x$.
This translates to a monotonicity in the domain: 
the larger the domain, the higher the probability that a given point is contained in a large loop. 
This fact is used to compare the loop~$O(1)$ model in a simply connected domain $\calD$ with 
that in the domain obtained from $\calD$ after removing certain interior parts. 
The latter is generally not simply connected, and it is essential that our monotonicity property can handle such domains.

\begin{question}\label{question1}
	Associate to the loop $O(n)$ model with edge weight $x$ in some domain $\calD$
	a positively associated percolation model $\Psi_{\calD, n,x}$ with the property that, 
	if one exhibits exponential decay of connection probabilities, then so does the other. 
\end{question}

The percolation model $\Psi_{\calD,n,x}$ actually only needs to have some monotonicity property in the domain, 
sufficient for our proof to apply. 
Unfortunately, we only have such an associated model when $n = 1$. 

Suppose that one may find such a model $\Psi$ for some value of $n$. Then our proof may be adapted. 
Indeed, fix~$x$ such that the loop $O(n)$ model exhibits exponential decay. 
Then $\Psi_{\calD,n,x}$ also exhibits exponential decay for any domain $\calD$. 
Consider now the loop $O(\tilde n)$ model with edge-weight $x$ for $\tilde n > n$ and colour each loop independently in red with probability~$n / \tilde n$ and in blue with probability $(\tilde n-n)/{\tilde n}$.
Then, conditionally on the blue loop configuration $\omega_b$, the red loop configuration has the law of the loop $O(n)$ model with edge-weight $x$ in the domain $\calD \setminus \omega_b$. 
By positive association, since~$\Psi_{\calD,n,x}$ exhibits exponential decay, so does~$\Psi_{\calD \setminus \omega_b,n,x}$.
Then the loop $O(\tilde n)$ model exhibits exponential decay of lengths of red loops and hence in general of lengths of all loops.

\bibliographystyle{abbrv}
\bibliography{biblicomplete2}

\begin{thebibliography}{10}

\bibitem{AizBarFer87}
M.~Aizenman, D.~J. Barsky, and R.~Fern{{\'a}}ndez.
\newblock The phase \mbox{transition} in a general class of {I}sing-type models
  is sharp.
\newblock {\em J. Statist. Phys.}, 47(3-4):343--374, 1987.

\bibitem{BefDum12}
V.~Beffara and H.~{Duminil-Copin}.
\newblock The self-dual point of the two-dimensional random-cluster model is
  critical for {$q\geq 1$}.
\newblock {\em Probab. Theory Related Fields}, 153(3-4):511--542, 2012.

\bibitem{BenHon16}
S.~Benoist and C.~Hongler.
\newblock The scaling limit of critical ising interfaces is {CLE}(3).
\newblock arXiv:1604.06975.

\bibitem{BloNie89}
H.~W. Bl{\"o}te and B.~Nienhuis.
\newblock The phase diagram of the {$O(n)$} model.
\newblock {\em Physica A: Statistical Mechanics and its Applications},
  160(2):121 -- 134, 1989.

\bibitem{CamNew06}
F.~Camia and C.~M. Newman.
\newblock Two-dimensional critical percolation: the full scaling limit.
\newblock {\em Comm. Math. Phys.}, 268(1):1--38, 2006.

\bibitem{ChaMac98}
L.~Chayes and J.~Machta.
\newblock Graphical representations and cluster algorithms ii.
\newblock {\em Physica A: Statistical Mechanics and its Applications},
  254(3):477 -- 516, 1998.

\bibitem{CheSmi12}
D.~Chelkak and S.~Smirnov.
\newblock Universality in the 2{D} {I}sing model and conformal invariance of
  fermionic observables.
\newblock {\em Invent. Math.}, 189(3):515--580, 2012.

\bibitem{CraGlaHarPel18}
N.~Crawford, A.~Glazman, M.~Harel, and R.~Peled.
\newblock Macroscopic loops in the loop~{$O(n)$} model via the {XOR} trick.
\newblock in progress, 2018.

\bibitem{DomMukNie81}
E.~Domany, D.~Mukamel, B.~Nienhuis, and A.~Schwimmer.
\newblock Duality relations and equivalences for models with {$O(n)$} and cubic
  symmetry.
\newblock {\em Nuclear Physics B}, 190(2):279--287, 1981.

\bibitem{DumGlaPel17}
H.~Duminil-Copin, A.~Glazman, R.~Peled, and Y.~Spinka.
\newblock Macroscopic loops in the loop {$O(n)$} model at {N}ienhuis' critical
  point.
\newblock 2017.
\newblock Preprint - arXiv:1707.09335.

\bibitem{DumPelSam14}
H.~{Duminil-Copin}, R.~Peled, W.~Samotij, and Y.~Spinka.
\newblock Exponential decay of loop lengths in the loop {$O(n)$} model with
  large $n$.
\newblock {\em Communications in {M}athematical {P}hysics}, 349(3):777--817, 12
  2017.

\bibitem{DumSmi12}
H.~{Duminil-Copin} and S.~Smirnov.
\newblock The connective constant of the honeycomb lattice equals
  {$\sqrt{2+\sqrt{2}}$}.
\newblock {\em Ann. of Math. (2)}, 175(3):1653--1665, 2012.

\bibitem{FriVel17}
S.~Friedli and Y.~Velenik.
\newblock {\em Statistical Mechanics of Lattice Systems: a Concrete
  Mathematical Introduction}.
\newblock Cambridge University Press, 2017.

\bibitem{GlaMan18}
A.~Glazman and I.~Manolescu.
\newblock Uniform lipschitz functions on the triangular lattice have
  logarithmic variations.
\newblock 2018.
\newblock Preprint --- arXiv:1810.05592.

\bibitem{Gri06}
G.~Grimmett.
\newblock {\em The random-cluster model}, volume 333 of {\em Grundlehren der
  Mathematischen Wissenschaften [Fundamental Principles of Mathematical
  Sciences]}.
\newblock Springer-Verlag, Berlin, 2006.

\bibitem{Hol74}
R.~Holley.
\newblock Remarks on the {FKG} inequalities.
\newblock {\em Comm. Math. Phys.}, 36:227--231, 1974.

\bibitem{KagNie04}
W.~Kager and B.~Nienhuis.
\newblock A guide to stochastic {L}{\"o}wner evolution and its applications.
\newblock {\em J. Statist. Phys.}, 115(5-6):1149--1229, 2004.

\bibitem{Len20}
W.~Lenz.
\newblock Beitrag zum {V}erst\"andnis der magnetischen {E}igenschaften in
  festen {K}\"orpern.
\newblock {\em Phys. Zeitschr.}, 21:613--615, 1920.

\bibitem{Nie82}
B.~Nienhuis.
\newblock Exact {C}ritical {P}oint and {C}ritical {E}xponents of {$O(n)$}
  {M}odels in {T}wo {D}imensions.
\newblock {\em Physical Review Letters}, 49(15):1062--1065, 1982.

\bibitem{Ons44}
L.~Onsager.
\newblock Crystal statistics. {I}. {A} two-dimensional model with an
  order-disorder transition.
\newblock {\em Phys. Rev. (2)}, 65:117--149, 1944.

\bibitem{PelSpi17}
R.~Peled and Y.~Spinka.
\newblock Lectures on the spin and loop {$O(n)$} models.
\newblock 2017.
\newblock Lecture notes --- arXiv:1708.00058.

\bibitem{Smi01}
S.~Smirnov.
\newblock Critical percolation in the plane: conformal invariance, {C}ardy's
  formula, scaling limits.
\newblock {\em C. R. Acad. Sci. Paris S{\'e}r. I Math.}, 333(3):239--244, 2001.

\bibitem{Smi10}
S.~Smirnov.
\newblock Conformal invariance in random cluster models. {I}. {H}olomorphic
  fermions in the {I}sing model.
\newblock {\em Ann. of Math. (2)}, 172(2):1435--1467, 2010.

\bibitem{Tag18}
L.~Taggi.
\newblock Shifted critical threshold in the loop {$O(n)$} model at arbitrary
  small~$n$.
\newblock 2018.
\newblock Preprint --- arXiv:1806.09360.

\end{thebibliography}

\end{document}